%% file: Real_Fibered_jag_rev_arxiv.tex
\documentclass{amsart}

\usepackage{cite}
\usepackage[all]{xy}
\usepackage{graphicx}
\usepackage{textcomp}
\usepackage{charter}
\usepackage[vflt]{floatflt}
\usepackage{amssymb}
\usepackage{float}
\usepackage{enumerate}
\usepackage[pdftex,colorlinks]{hyperref}

\newtheorem{thm}{Theorem}[section]
\newtheorem{prop}[thm]{Proposition}
\newtheorem{lem}[thm]{Lemma}
\newtheorem{cor}[thm]{Corollary}
\newtheorem{quest}[thm]{Question}

\theoremstyle{definition}
\newtheorem{dfn}[thm]{Definition}

\newtheorem{example}[thm]{Example}
\newtheorem{ex}[thm]{Example}
\newtheorem{rem}[thm]{Remark}

\theoremstyle{remark}

\newtheorem{remark}[thm]{Remark}

\numberwithin{equation}{section}

\include{operators}
\include{fonts}

\begin{document}

\title{Real Fibered Morphisms and Ulrich Sheaves}

\author{Mario Kummer}
\address{Technische Universit\"at Berlin\\
Berlin, Germany}
\email{kummer@tu-berlin.de}
\author{Eli Shamovich}
\address{Department of  Mathematics\\ 
Ben-Gurion University of the Negev\\
84105 Beer-Sheva, Israel}
\email{eshamovich@uwaterloo.ca}
\thanks{The research of E. S. was partially carried out during the visits to the Department of Mathematics and Statistics of the University of Konstanz, supported by the EDEN Erasmus Mundus program (30.12.2013 - 30.6.2014). M. K. was supported by the Studienstiftung des Deutschen Volkes.}
\begin{abstract}
In this paper, we define and study real fibered morphisms. Such morphisms arise in the study of real hyperbolic hypersurfaces in $\pp^d$ and other hyperbolic varieties. We show that real fibered morphisms are intimately connected to Ulrich sheaves admitting positive definite symmetric bilinear forms.
\end{abstract}
\maketitle
\tableofcontents

$\,$

\section{Introduction}
\subsection{Background}
A homogeneous polynomial $f \in \R[x_0,\ldots,x_d]$ is called hyperbolic with respect to a point $e \in \R^{d+1}$, if $f(e) \neq 0$ and for every $x \in \R^{d+1}$ the roots of the univariate polynomial $f(e + t x)$ are all real. Hyperbolic polynomials were first studied in the context of partial differential equations since they arise as symbols of hyperbolic (and hence the name) partial differential equations with constant coefficients. Such PDEs are of interest because the Cauchy problem is well-defined in this case, see for example \cite{Ga51} and \cite{HorII}. The first to study geometric properties of hyperbolic polynomials was G{\aa}rding in \cite{Ga59}. G{\aa}rding showed that hyperbolic polynomials possess remarkable convexity properties and those results were extended by Bauschke, G{\"u}ler, Lewis and Sendov in \cite{BGLS01}.
In the last years, there has also been ample interest in hyperbolic polynomials from the areas of combinatorics \cite{COS04} and optimization \cite{Gu97, Re06}.
In that context, the so-called generalized Lax conjecture is an important open question asking whether the feasible sets of hyperbolic programming are the same as the feasible sets of semidefinite programming.
More recently, properties of stable polynomials, a special kind of hyperbolic polynomials, and certain determinantal representations of them were crucially used in the proof of the Kadison--Singer conjecture by Marcus, Spielman, and Srivastava \cite{MSS14}. Br\"and\'en reproved and slightly strengthened their results using convexity properties of hyperbolic polynomials \cite{Bra14}.
Very recently, hyperbolic polynomials also appeared in the context of exponential families in statistics \cite{MSUZ14}.

One can reformulate the hyperbolicity property in a more geometric way, namely, consider the hypersurface $X \subset  \pp^d$ cut out by $f$ and consider a point $e \in \pp^d$ off $X$. Then $f$ is hyperbolic with respect to $e$ if and only if for every real line $L$ through $e$ we have that $L \cap X \subset  X(\R)$. Vinnikov and the second author generalized this idea in \cite{SV14} to define hyperbolicity of a general real subvariety of $\pp^d$ with respect to a real linear subspace of correct dimension.
Well-studied examples of hyperbolic varieties that are not hypersurfaces are reciprocal linear spaces, i.e, the Zariski closure of the Cremona transform of a linear subspace in projective space. These varieties have been examined for example in the context of interior points methods for linear programming \cite{DLSV12} and entropy maximization for log-linear models \cite{SSV13}. Hyperbolicity (though not called so) of these varieties was shown by Varchenko \cite{Var95} and is used (not just) in the cited works at various points.

The classical example of a hyperbolic polynomial is the determinant of a generic symmetric $n \times n$ matrix. It can be shown that this polynomial is hyperbolic with respect to the identity matrix. This led Lax in 1958 \cite{Lax58} to ask whether every hyperbolic homogeneous polynomial in three variables has a determinantal representation. More precisely, assume $f \in \R[x_0,x_1,x_2]$ is a homogeneous polynomial of degree $m$, hyperbolic with respect to $(1,0,0)$. Lax asked whether there exist symmetric matrices $A_0, A_1, A_2 \in \textrm{M}_m(\R)$, with $A_0$ positive definite and $f = \det(x_0 A_0 + x_1 A_1 + x_2 A_2)$. It was observed by Lewis,Parrilo and Ramana in \cite{LPR05} that this follows from a result of Vinnikov and Helton in \cite{HelVin07}. The conjecture fails for $d > 2$ even in a weakened form, see \cite{Bea00}, \cite{Bra11} and \cite{KerVin12} for more details. The generalized Lax conjecture described above can be formulated as follows. Given a homogeneous 
polynomial $f \in \R[x_0,\ldots,x_d]$ hyperbolic with respect to $e$, can we find another homogeneous polynomial $g \in \R[x_0,\ldots,x_d]$ hyperbolic with respect to $e$, such that the product $fg$ is hyperbolic with respect to every point in the connected component of $e$ in $\R^{d+1} \setminus Z(f)$ and $fg$ has a symmetric determinantal representation definite at $e$? The best result known today regarding this conjecture is due to the first author in \cite{Kum}. The reader is referred to \cite{Vppf} for an extensive overview of classical notions of hyperbolicity and determinantal representations.

The goal of this paper is to study hyperbolicity and determinantal representations in an invariant way. Let $X \subset  \pp^d$ be a real subvariety hyperbolic with respect to a linear subspace $V \subset  \pp^d$, then the linear projection with center $V$ defines a morphism over the reals from $X$ to $\pp^k$ (here $k = \dim X$) that sends only real points to real points. Following the idea of Grothendieck, we isolate this property of a morphism between real (projective) varieties. In Section \ref{sec:real_fibered} we define real fibered morphisms as (finite and surjective) morphisms that map only real points to real points. We prove that real fibered morphisms are always unramified at smooth real points. We conclude that the Veronese embedding of $\pp^k$ is not hyperbolic whenever $k \geq 2$. We say that $X$ is weakly hyperbolic if $X$ admits a real fibered morphism to $\pp^k$. In the case of smooth projective curves, this is equivalent to the corresponding Riemann surface being of dividing type (sometimes also called type one). We then show that every such curve admits a hyperbolic embedding into $\pp^3$. Therefore, we conclude that every weakly hyperbolic curve can be hyperbolically embedded. This, however, fails if $\dim X > 1$ and we provide an example of a weakly hyperbolic variety that admits no hyperbolic embeddings.

While in the context of hyperbolicity real fibered morphisms appear as linear projections there are also applications where this is not the case. For example, the existence of real fibered morphisms from curves to the projective line and their properties have been studied by several authors \cite{huisman, gabard, cophuis, cop}. Typical questions include but are not limited to general existence results \cite{gabard} or constructions of real fibered morphisms of small degree \cite{cop}. In the study of amoebas of algebraic varieties, the question arises for which real hypersurfaces the logarithmic Gauss map, i.e. the map that sends a point $(x_0:\cdots:x_n)$ on the hypersurface defined by $h$ to $(x_0\frac{\partial h}{\partial x_0}:\cdots:x_n\frac{\partial h}{\partial x_n})$, is real fibered \cite{PR10, Mik04}. In fact, it was recently shown in \cite{brugalle2015non} that such hypersurfaces must have singularities by using our results on real fibered morphisms\footnote{The result in \cite{brugalle2015non} was obtained after a preliminary version of this paper was published on arXiv.}.

After reminding the reader of some results from commutative algebra in Section \ref{sec:comm}, we recall the definition of determinantal representations from \cite{SV14} and the notion of Ulrich sheaves from \cite{EFS03} and \cite{ESW03} in Section \ref{sec:ulrich}. Ulrich modules were first studied by Ulrich and his collaborators, see for example \cite{BHU87} and \cite{HUB91}. Eisenbud and Schreyer used Ulrich sheaves to construct determinantal and Pfaffian representations of Chow forms of subvarieties of $\pp^d$. We show that in fact Ulrich sheaves can be identified with determinantal representations in the sense of \cite{SV14}. This shows, in particular, that every determinantal representation of a smooth projective curve can be obtained using the algorithm described in \cite{SV14}. We believe that the correspondence between Ulrich sheaves and determinantal representations can be used in the future for proving the existence of Ulrich sheaves on certain subvarieties of $\pp^d$, as it has been done in \cite{kummer2016chow} for reciprocal linear spaces. The question of which subvarieties of $\pp^d$ can be the support of an Ulrich sheaf is of particular interest since the Boij--{S}\"oderberg cone of such varieties is the same as the one of projective space \cite{boij}. 

In Section \ref{sec:real_bilinear} we show that definite symmetric determinantal representations can be defined in terms of positive definite bilinear forms on Ulrich sheaves. This result has been used in \cite{kummer2016chow} to answer a question from \cite{SSV13}.
Following \cite{KMS14} we define the relative notion of $f$-Ulrich sheaves with respect to a finite flat morphism 
$f \colon X \to Y$. We give a characterization of real fibered morphisms in terms of positive semidefinite bilinear forms on coherent sheaves which
generalize the classic methods of checking real rootedness of univariate polynomials like the Hermite matrix or the B\'ezout matrix.
In particular, we show that if $f \colon X \to Y$ is a finite flat morphism with a positive $f$-Ulrich sheaf on $X$, 
then $f$ is real fibered. We then proceed to formulate a question which can be considered as a relative version of the generalized Lax conjecture.

\subsection{Notations and Convention}

Let $\mathbb{K}$ stand for either the field of real or complex numbers. By a $\mathbb{K}$-variety we mean a reduced, separated scheme of finite type over $\Spec \mathbb{K}$,
not necessarily irreducible.
A morphism of $\mathbb{K}$-varieties is always meant to be a morphism over $\Spec \mathbb{K}$.
A curve over $\mathbb{K}$ is a $\mathbb{K}$-variety of pure dimension one.
Now let $X$ be an $\R$-variety. 
We will write $X(\mathbb{K}) = \Hom_{\Spec \R}(\Spec \mathbb{K}, X)$ for the set of $\mathbb{K}$-points.
We can identify $X(\R)$ with the set 
of points $x \in X$, such that the residue field $\kappa(x)$ of $X$ at $x$ is $\R$. We will write 
$X_\C=X \times_{\Spec \R} \Spec \C$ for the complexification. 

The following description appears in \cite[Sec.\ 1.1]{Sil89} and in \cite[Ex.\ II.4.7]{Har77}. For every quasi-projective $\R$-variety $X$  the complexification $X_\C$ of $X$ comes equipped with an involution $\tau$, i.e.,
an isomorphism of $\R$-varieties $\tau: X_\C \to X_\C$,
such that $X(\R)$ can be identified with $X_{\C}(\C)^{\tau} = X(\C)^{\tau}$, namely the fixed points of $\tau$.
Given two $\R$-varieties $X$ and $Y$ and a morphism $f \colon X \to Y$, we have a corresponding morphism $f_{\C} \colon X_{\C} \to Y_{\C}$ and if we denote the involution of $X$ resp. $Y$ by $\tau$ resp. $\sigma$ we get that $f_{\C} \circ \tau = \sigma \circ f_{\C}$. Conversely, every morphism $g_\C \colon X_\C \to Y_\C$ that intertwines the involution, i.e., $g_{\C} \circ \tau = \sigma \circ f_{\C}$, descends to a morphism $g: X \to Y$.

We will use various results from real algebraic geometry. Good introductory references are \cite{And96, BCR13, MarshSOS, PosPols, ERA}.
Occasionally, we will make use of the \textit{real spectrum} of a ring. Given a ring $A$ the real spectrum $\Sper A$ is the set
of all pairs $\alpha=(\mathfrak{p},P)$ where $\mathfrak{p}$ is a prime ideal of $A$ and $P$ is an ordering of the residue field $\kappa(\mathfrak{p})$ \cite[\S 7.1]{BCR13}. See \cite[Prop. 7.1.2]{BCR13} for equivalent definitions.
Let $\alpha=(\mathfrak{p},P) \in \Sper A$ and let $\rho_{A, \mathfrak{p}}: A \to \kappa(\mathfrak{p})$ be
the canonical homomorphism.
We denote by $\Supp(\alpha)=\mathfrak{p}$ the support of $\alpha$, i.e., the prime ideal corresponding to $\alpha$.
For any element $f \in A$ we say that $f(\alpha) \geq 0$, i.e., $f$ is nonnegative in $\alpha$, if $\rho_{A, \mathfrak{p}}(f) \in P$.
We write $f(\alpha)>0$ if $f(\alpha)\geq 0$ and $f \not\in \Supp(\alpha)$.
On $\Sper A$ we consider the \textit{spectral topology}. This is the topology on $\Sper A$ that is generated by the subbasis of open sets of the form \[\{\alpha \in \Sper A:\,\, f(\alpha)>0\}\] for $f \in A$ \cite[Def. 7.1.3]{BCR13}.
Note that every ring homomorphism $A \to B$ induces a continuous map $\Sper B \to \Sper A$ \cite[Prop. 7.1.7]{BCR13}.
Now let $A$ be a finitely generated reduced $\R$-algebra and let $X=\Spec A$ be the corresponding affine $\R$-variety. Since $\R$ has exactly one ordering, the points of $\Sper A$ whose support is a maximal ideal can be identified with $X(\R)$.
Under this identification, we have that $X(\R)$ is dense in $\Sper A$ \cite[III \S 3, Thm. 7]{ERA}.
Finally, for $\alpha, \beta \in \Sper A$ we say that $\alpha$ specializes to $\beta$ if $\beta \in \overline{\{\alpha\}}$.

\bigskip

\noindent \textbf{Acknowledgements.}
We would like to thank Christoph Hanselka, Claus Scheiderer, Bernd Sturmfels, Victor Vinnikov and Amnon Yekutielli
for some helpful discussions related to the subject of this paper. 

\section{Real Fibered Morphisms} \label{sec:real_fibered}

In this section, we work over the ground field $\R$. We will write $\pp^d = \pp_\R^d$ for the projective $d$-space
over $\R$.
The goal of this section is to study real fibered morphisms of $\R$-varieties. We start with a definition.

\begin{dfn} \label{dfn:real_fibered}
Let $f \colon X \to Y$ be a finite surjective morphism of $\R$-varieties.
We say that $f$ is real fibered if, for every $x \in X$, we have that $f(x) \in Y(\R)$ if and only if $x \in X(\R)$.
\end{dfn}

\begin{rem}
 It is clear that every $\R$-point is sent to an $\R$-point. The real fibered property implies the converse as well.
 Now let $X$ and $Y$ be quasi-projective $\R$-varieties and $g_{\C} \colon X_{\C} \to Y_\C$ be a morphism that intertwines the involutions on $X_\C$ and $Y_\C$ which is finite and surjective. Then by descent theory the morphism of $\R$-varieties $g \colon X \to Y$
 that we get is real fibered if and only if $g_{\C}$ maps only 
 fixed points (of the involution) to fixed points. For details on Galois descent see \cite[Sec.\ 14.20]{GW10}
\end{rem}

\begin{rem}
The property of being real fibered is stable under base-change, i.e., if $X \to Y$ is real fibered and $Y' \to Y$ is any morphism of $\R$-varieties, then the induced morphism $X \times_Y Y' \to Y'$ is also real fibered. Furthermore, the composition of two real fibered morphisms is again real fibered.
\end{rem}

To motivate the definition we recall the following definition of hyperbolic varieties from \cite{SV14}.

\begin{dfn} \label{dfn:hyperbolic}
Let $X$ be a projective $\R$-variety of dimension $k$ and let $\iota \colon X \to \pp^d$ be an embedding.
Let $V \subset  \pp^d$ be a real $d-k - 1$-dimensional linear subspace, such that $\iota(X) \cap V = \emptyset$. We say that $\iota$ is a hyperbolic embedding if for every real $d-k$-dimensional subspace $U$ containing $V$, we have $U \cap \iota(X) \subset  \iota(X(\R))$. If the embedding is clear from the context, we will simply say that $X$ is hyperbolic with respect to $V$.
\end{dfn}

Assume that $X\subset \pp^d$ is hyperbolic with respect to $V$.
Then the linear projection from $V$ induces a finite surjective morphism $f \colon X \to \pp^k$.
Furthermore, $X$ being hyperbolic with respect to $V$ implies that $f$ is real fibered. Therefore, we make the following definition.

\begin{dfn} \label{dfn:weakly_hyperbolic}
We say that an $\R$-variety $X$ is weakly hyperbolic if there exists a real fibered morphism $f \colon X \to \pp^k$.
\end{dfn}


\begin{example}\label{exp:interlace}
 Let $C \subset  \pp^n$ be the rational normal curve of degree $n$, i.e., the image of the morphism
 \[ \pp^1 \to \pp^n, \,\, (s:t) \mapsto (s^n: s^{n-1} t : \ldots : t^n). \]
 The purpose of this example is to show that this is a hyperbolic embedding of $\pp^1$ and to give a description of all the $n-2$-dimensional subspaces of $\pp^n$ with respect to which $C$ is hyperbolic.
 
 Consider a linear projection of $C$ from an $n-2$-plane in $\pp^n$ disjoint from $C$ to $\pp^1$.
 After choosing a basis on $\pp^1$ this corresponds to a morphism $\pp^1 \to \pp^1$ given by two 
 bivariate polynomials $f$ and $g$ of degree $n$ without common projective zero. The elements of the fiber over a real point $(\lambda: \mu) \in \pp^1$ under this morphism are the zeros of $\mu  f - \lambda g$.
 Thus the morphism is real fibered if and only if for all $\lambda, \mu \in \R$, not both zero, the polynomial $ \mu  f - \lambda g$ has only  real roots. It is classically known, see, for example, \cite[Thm. 6.3.8]{Rah02},
 that this is equivalent to $f$ and $g$ having interlacing zeros, i.e., all zeros
 of $f$ and $g$ are simple and real and each connected component of $\pp^1(\R)\smallsetminus\{P \ | \, f(P)=0\}$ contains exactly one zero of $g$ and vice versa.

 To make an explicit example consider the case $n=3$, the case of the twisted cubic curve.
 The line $L_1$ in $\pp^3$ spanned by the points $(0:1:0:0)$ and $(0:0:1:0)$ is disjoint from $C$ and the projection from $L_1$ corresponds to the morphism $\pp^1\to \pp^1, \, (s:t) \mapsto (s^3:t^3)$. The zeros of $s^3$ and $t^3$ clearly do not interlace and we have that $C$ is not hyperbolic with respect to $L_1$. But the line $L_2$ spanned by the points $(1:0:1:0)$ and $(0:1:0:4)$ corresponds to the morphism $\pp^1\to \pp^1, \, (s:t) \mapsto (s^3-st^2:4s^2 t-t^3)$. Since the zeros of $s^3-st^2$ and $4s^2 t-t^3$ do interlace, the twisted cubic $C$ is hyperbolic with respect to $L_2$.
 
 Hyperbolicity of the rational normal curve can also be seen without using results about interlacing polynomials.
 The M\"obius transformation  $\Phi(z)=\frac{z-i}{z+i}$ sends the real line (including infinity) to the unit circle.
 Conversely, every point that is sent to the unit circle lies on the real line. 
 The map $\psi_k(z)=z^k$ has the property that $z$ is on the unit circle 
 if and only if $\psi_k(z)$ is on the unit circle for $k \geq 1$.
 This shows that the morphism $\pp^1 \to \pp^1$ that we get from $\Phi^{-1} \circ \psi_k \circ \Phi$
 is a real fibered morphism of degree $k$ for all $k \geq 1$.
\end{example}

The following proposition gives a condition when a weakly hyperbolic variety admits a hyperbolic embedding.

\begin{prop} \label{prop:very_ample_hyperbolic}
Assume that $X$ is weakly hyperbolic, with $f \colon X \to \pp^k$ being the real fibered morphism. If $\cL = f^* \cO_{\pp^k}(1)$ is very ample, then $X$ admits a hyperbolic embedding.
\end{prop}
\begin{proof}
Let $\iota \colon X \to \pp(\textrm{H}^0(X,\cL)^*)$ be the embedding obtained from $\cL$. Write $f =(\mu_0:\cdots:\mu_k)$, for $\mu_0,\ldots,\mu_k \in \textrm{H}^0(X,\cL)$. Then it is immediate that $\iota(X)$ is hyperbolic with respect to the real subspace $\mu_0 = \ldots = \mu_k = 0$.
\end{proof}

For any $\R$-variety $X$, we can equip $X(\C)$ with the classical topology and $X(\R)$ is a 
closed subset of $X(\C)$ with respect to the classical topology.
Recall that a smooth, geometrically irreducible projective curve $X$ is called of dividing type if $X(\C) \setminus X(\R)$ has two connected components. Then we have the following:

\begin{thm} \label{thm:curves_embedding}
Let $X$ be a smooth, geometrically irreducible projective curve over $\R$. The following are equivalent:
\begin{enumerate}[(i)]
\item $X$ is of dividing type.

\item $X$ is weakly hyperbolic.

\item $X$ admits a hyperbolic embedding into some $\pp^d$.

\end{enumerate}

\end{thm}
\begin{proof}
It is immediate that $(iii)$ implies $(ii)$.
The direction $(i)\Rightarrow(ii)$ is \cite[Thm. 7.1]{gabard}, see also \cite[\S 4.2]{Ahl50} for the original proof in the analytic setup. On the other hand, if $f: X \to \pp^1$ is a real fibered morphism, then $f(X(\C)\setminus X(\R))=\pp^1(\C)\setminus\pp^1(\R)$. Since $\pp^1(\C)\setminus\pp^1(\R)$ is not connected, it follows that $X(\C)\setminus X(\R)$ is not connected as well and hence $(i)$ is equivalent to $(ii)$.

Now let $f: X \to \mathbb{P}^1$ be real fibered and $\cL = f^* \cO_{\pp^1}(1)$. Note that $\cL$ is an ample line bundle on $X$ since $f$ is finite. In Example \ref{exp:interlace} we have seen that for any $n\geq 1$ there is a real fibered morphism $g: \pp^1 \to \pp^1$, such that $\cO_{\pp^1}(n) = g^* \cO_{\pp^1}(1)$ and for sufficiently large $n$ the line bundle $\cL^{n}=(g \circ f)^* \cO_{\pp^1}(1)$ is very ample. Thus by Proposition \ref{prop:very_ample_hyperbolic} $(ii)$ implies $(iii)$.
%
\end{proof}

\begin{rem}
 Theorem \ref{thm:curves_embedding} says that every weakly hyperbolic curve admits a hyperbolic embedding into some projective space. We will see in Example \ref{exp:doublecover} that this is not true for higher-dimensional varieties.
\end{rem}

\begin{cor} \label{cor:curves_embedding_low_dim}
 Let $X$ be a smooth projective curve of dividing type.
 Then $X$ admits a hyperbolic embedding into $\pp^3$ and a birational hyperbolic embedding into $\pp^2$.
\end{cor}
\begin{proof}
Assume that we can embed $X$ in $\pp^d$, such that the image is hyperbolic with respect to some $d-2$-dimensional real subspace $V \subset  \pp^d$. The tangent variety to $X$ is of dimension at most two and the secant variety is of dimension at most three. Thus if $d > 3$ we can find a real point in $V$ disjoint from the secant variety and project from it. In \cite[Thm.\ 3.10]{SV14} states that if $X$ is hyperbolic with respect to $V$, then there exists an open subset (in the classical topology) of the Grassmannian $\G(d-2,d)(\R)$ containing $V$, such that $X$ is hyperbolic with respect to any subspace in this subset. Thus we can perturb $V$ slightly if needed. We obtain an embedding of $X$ into $\pp^{d-1}$ hyperbolic with respect to the image of $V$. Hence we can embed $X$ hyperbolically into $\pp^3$.
We can repeat this argument in case $d = 3$ to get a finite map from $X$ into $\pp^2$, birational onto its image.
\end{proof}

\begin{rem}
 Let $X$ be a projective, geometrically irreducible, smooth, real curve. Let $g$ be its genus and let $s$
 be the number of connected components of $X(\R)$. If $s=g+1$, then $X$ will be of dividing type.
 If $X$ is of dividing type, then $g+1-s$ will be even \cite[\S 21]{Kl23}.
 \end{rem}
 
\begin{example}
  The TV-Screen is the plane quartic curve defined by $x^4+y^4-z^4$. Its genus is three and $X(\R)$ has only one connected component. Thus $g+1-s=3$ is odd and therefore the TV-Screen admits no hyperbolic embedding.
\end{example}

\begin{example}
 The Edge quartic is the plane quartic curve defined by \[25\cdot (x^4+y^4+z^4)-34\cdot (x^2y^2+x^2z^2+y^2z^2).\] The set of its real points has four connected components. Thus it is not hyperbolic with respect to any point in the plane. But it can be embedded hyperbolically into some $\pp^d$ since it is of dividing type. In fact, we can describe such an embedding concretely. Each of the ovals bounds a region in $\pp^2(\R)$ that is homeomorphic to a two-dimensional disc. Fix a point in the interior of each of these four regions and consider the pencil of quadrics that pass through all these four points. Each such quadric will intersect the curve in eight real points. Thus we get a real fibered morphism to $\pp^1$. This corresponds to a hyperbolic embedding of the Edge quartic into $\pp^5$ via the second Veronese embedding of $\pp^2$. It is hyperbolic with respect to the three-dimensional subspace of $\pp^5$ that is spanned by the image of our four chosen points. In order to explicitly compute a birational hyperbolic embedding into $\pp^2$, we choose $(-1:1:1),(1:-1:1),(1:1:-1)$ and $(1:1:1)$ as our four points. The three quadrics $$ x y + x z + y z+x^2 , x y + x z + y z + y^2, x y  + x z + y z+ z^2$$ are a basis of all quadrics vanishing on $(-1:1:1),(1:-1:1),(1:1:-1)$. Thus the image of the Edge quartic under the map to $\pp^2$ defined by those three quadrics is a plane hyperbolic curve. It is the zero set of the following symmetric polynomial: $$27\cdot(x^4 y^4+x^4 z^4+y^4 z^4) - 36\cdot(x^4 y^3 z +x^3 y^4 z+x^4 y z^3+ x y^4 z^3  +x^3 y z^4+x y^3 z^4)$$ $$ - 382\cdot(x^4 y^2 z^2+x^2 y^4 z^2+x^2 y^2 z^4) +      436\cdot(x^3 y^3 z^2   +    x^3 y^2 z^3 + x^2 y^3 z^3) .    $$

\end{example}
\begin{figure}[h]
 \includegraphics[width=4cm]{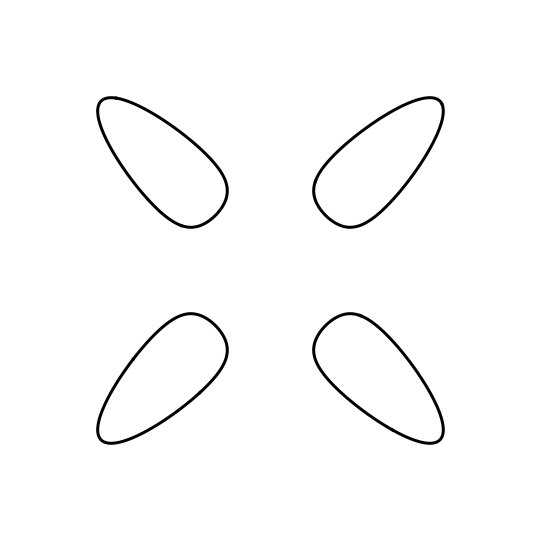} \quad
 \includegraphics[width=4cm]{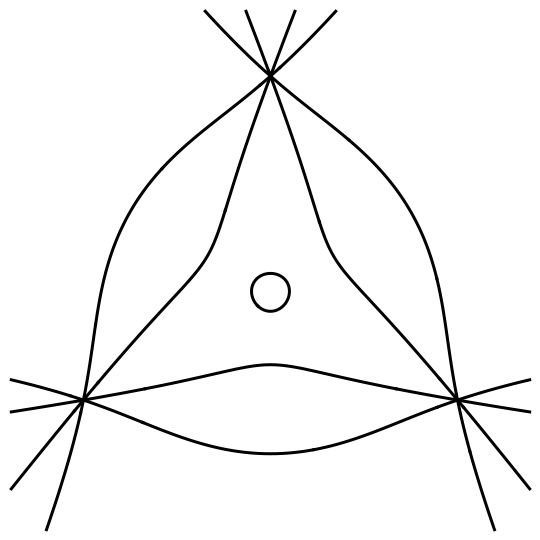}
\caption{The Edge quartic (on the left) and a birational hyperbolic embedding of it into $\pp^2$.}
\label{fig:edgecurves}
\end{figure}

 Let $B$ be a ring and let $A$ be a $B$-algebra which is a finitely generated free $B$-module. Any element $a\in A$ defines the $B$-endomorphism $m_a: A\to A, \, b\mapsto a\cdot b$. We define $\textrm{tr}(a)\in B$ to be the trace of the endomorphism $m_a$.

 Now let $B=\R$ and $A$ be an $\R$-algebra which is a finite-dimensional $\R$-vector space.
 Recall from \cite[Thm. 2.1]{counting} that then the trace
 bilinear form $A \times A \to \R, (f,g) \mapsto \textrm{tr}(fg)$
 is positive semidefinite if and only if $\Spec(A)$ consist only of $\R$-points.
 Similarly, if $K \subset  L$ is a field extension of degree $m$ and $P$ an ordering of $K$,
 then the signature (with respect to $P$) of the trace bilinear form $L \times L \to K, (f,g) \mapsto \textrm{tr}(fg)$ is the number of different extensions
 of $P$ to $L$, see also \cite[\S 8]{ERA}.
To characterize real fibered morphisms we have the following theorem:

\begin{thm} \label{thm:function_field_orede	}
Let $X$ and $Y$ be irreducible $\R$-varieties. Assume that $Y$ is smooth. Let $f \colon X \to Y$ be a finite, flat and surjective morphism, then the following are equivalent:
\begin{enumerate}[(i)]
\item The morphism $f$ is real fibered.
\item Every ordering of the function field $K$ of $Y$ has exactly $m$ extensions to the function field $L$ of $X$, where $m = [L:K]$ the degree of $f$.
\end{enumerate}
\end{thm}
\begin{proof}
Consider the $K$-bilinear form $b: L \times L \to K, \, (f,g) \mapsto \tr_{L/K} (f \cdot g)$. 
  For every point $y \in Y(\R)$ there exists an open affine  neighborhood $U \subset  Y$ of 
$y$, such that $f^{-1}(U) \subset  X$ is  affine and $\mathcal{O}_{X}(f^{-1}(U))$  
is a finite free module over $\mathcal{O}_{Y}(U)$. Let $A=\mathcal{O}_{Y}(U)$ and $B=\mathcal{O}_{X}(f^{-1}(U))$. Therefore, the trace map 
$\textrm{tr}_{L/K}: L \to K$ satisfies $\textrm{tr}_{L/K}(B) \subset  A$.  Now $(i)$ together with the above remark implies that the $A$-bilinear 
form  \[\mathfrak{b}: B \times B \to A, 
\,\, (a,b) \mapsto \textrm{tr}_{L/K}(ab)\]  is positive semidefinite at every point from $U(\R)$, 
thus it is positive semidefinite on $\Sper(A)$. Indeed, $U(\R)$ is dense in $\Sper(A)$ \cite[III \S 3, Thm. 7]{ERA} and therefore the closed set defined by the principal minors of the matrix associated to the bilinear form is everything.
In particular, the $K$-bilinear form  
\[b: L \times L \to K, \,\, (a,b) \mapsto \textrm{tr}_{L/K}(ab)\] is positive definite on $\Sper(K) \subset  \Sper(A)$, i.e., the signature of $b$ is $m$ for every ordering of $K$. By what has been said above this implies $(ii)$. 
In order to prove $(ii) \Rightarrow (i)$, assume that $f^{-1}(\{y\}) \not\subset  X(\R)$.  This means that the bilinear 
form $\mathfrak{b}$ is not positive semidefinite in $y$. Since $y$ is a smooth point of $Y$, 
the Artin--Lang Theorem \cite[Thm. 1.3]{Bec82} implies that $b$ is not positive semidefinite on $\Sper(K)$. 
\end{proof}

\begin{remark}
 The previous theorem is true even without the assumption of flatness \cite[Thm. 2.4.5]{kummerdiss} but for the sake of simplicity and since we will apply the theorem only to flat morphisms we do not give a proof here.
\end{remark}

For the following recall that we assume all curves to be equidimensional.

\begin{cor} \label{cor:real_fibered_normalization}
Let $f \colon X \to Y$ be a finite surjective morphism of curves over $\R$.
Assume that $Y$ is smooth, then $f$ is real fibered if and only if $f \circ \pi$ is real fibered, where $\pi \colon \tilde{X} \to X$ is the normalization map.
\end{cor}
\begin{proof}
Without loss of generality, we can assume that $X$ and $Y$ are irreducible.
Since all $\R$-varieties of dimension one are Cohen--Macaulay \cite[Ex. 2.1.20]{BrHr93} we have that both $f$ and $f \circ \pi$ are flat. Now since the function fields of $X$ and of $\tilde{X}$ are the same, the claim follows immediately from the above theorem.
\end{proof}

\begin{prop} \label{prop:real_fibered_unrmf_curves}
Let $f \colon X \to Y$ be a finite surjective morphism of curves over $\R$. Let $p \in X(\R)$ be such that $Y$ is smooth at $f(p)$. If the differential $\diff_p f \colon \textnormal{T}_p X \to \textnormal{T}_{f(p)} Y$ at $p$ is the zero map, then $f$ is not real fibered.
\end{prop}
\begin{proof}
Without loss of generality we can assume that $Y$ is smooth and that $X$ and $Y$ are irreducible and affine. Let $\pi: \tilde{X} \to X$ be the normalization map and let $q \in \tilde{X}(\R)$ be any point, such that $\pi(q)=p$ (we can assume that such a point $q$ exists since otherwise $f \circ \pi$ and thus $f$ would fail to be a real fibered morphism).  We have $\diff_q(f \circ \pi)=\diff_p f \circ \diff_q \pi=0$.  Thus by the preceding corollary, we can further restrict to the case where $X$ is smooth.  

Let $X=\Spec(A)$ and $Y=\Spec(B)$ and let $\tilde{f}: \Sper(A) \to \Sper(B)$ be the induced map between the real  spectra. Let $d$ be the degree of $f$.  By the Baer--Krull Theorem \cite[Thm. 2.2.5]{Eng05} there are two distinct points  $\alpha_1, \alpha_2 \in\Sper(B)$ with  support zero that specialize to $f(p)$.   If $f$ is real fibered the theorem above implies that there are $2d$ distinct points in the preimage    $\tilde{f}^{-1}(\{\alpha_1, \alpha_2\})$.   By real going-up \cite[Thm. 4.3]{And96} these specialize to points in the preimage of $f(p)$.    But since $f$ is ramified at $p$, there are at most   $d-1$ points in the preimage of $f(p)$. By the pigeonhole principle, we thus have at least one point   in $\Sper(A)$ to which at least three distinct points with support zero specialize. But since   $X$ is smooth  this contradicts the Baer--Krull Theorem. 
\end{proof}

\begin{rem}
The fact that a real fibered morphism $f \colon X \to \pp^1$, where $X$ is a smooth curve over $\R$, is unramified at real points can be easily seen using complex analysis. Consider $f$ as a meromorphic function on $X(\C)$ and consider its Laurent expansion in some real local coordinate around a zero to see that it has to be a simple zero.
\end{rem}

The following theorem has been proved in several special cases like for hyperbolic hypersurfaces \cite{HelVin07} or reciprocal linear spaces \cite{SSV13}.
However, their methods do not generalize to the case of arbitrary real fibered morphisms.

\begin{thm} \label{thm:real_fibered_unrmf}
Let $f \colon X \to Y$ be a real fibered morphism between two $\R$-varieties.
Let $p \in X(\R)$ and $q=f(x)\in Y(\R)$ be smooth points.
Then the differential $\diff_p f \colon \textnormal{T}_p X \to \textnormal{T}_{q} Y$ at $p$ is an isomorphism.
\end{thm}
\begin{proof}
 Assume  that the differential $\diff_p f$ of $f$ at $p$ is not surjective. Let $C \subset  Y$ be a curve over $\R$ which is smooth at $q$ and whose tangent space intersects the  image of $\diff_p f$ trivially. Let $C' =X \times_Y C$ be the fiber of $f$ over $C$ which is again a curve. The induced map $ C' \to C$ is real fibered and its differential at $p$ is zero. This contradicts the preceding lemma. 
\end{proof}
%

The following corollary is a partial generalization of \cite[Thm.\ 5.2]{HelVin07}.

\begin{cor} \label{cor:real_point_structure}
Assume that $X$ is a smooth weakly hyperbolic $\R$-variety of dimension $k \geq 2$ with $f \colon X \to \pp^k$ being a real fibered morphism. 
Then $X(\R)$ is a disjoint union of $s$ components homeomorphic to 
$S^k$ and $r$ components homeomorphic to $\pp^k(\R)$ where $2s+r=\deg f$ (both with respect to the classical topology on $X(\R)$).
In particular, every other real fibered morphism $X \to \pp^k$ has to be of degree $\deg f$.
\end{cor}

\begin{proof}
By the above theorem $X(\R)$ is a covering space of $\pp^k(\R)$ with $\deg f$ many sheets. Since $ k \geq 2$ we have that $\pi_1(\pp^k(\R)) \cong \Z/2$. 
Hence every connected component of $X(\R)$ is either homeomorphic to $S^k$ or to $\pp^k(\R)$ and the formula $2s+r=\deg f$
follows from counting the sheets.
\end{proof}

We have seen in Example \ref{exp:interlace} that the rational normal curve is always hyperbolic. This fails for Veronese varieties of higher dimension.

\begin{cor} \label{cor:veronese_not_hyperbolic}
If $X = \cV_{m}(\pp^k)$ is the Veronese embedding of $\pp^k$ ($k,m \geq 2$) into $\pp^N$, where $N = \binom{ k + m}{  m} -1$, then $X$ is not hyperbolic, i.e., there is no real linear subspace $V \subset  \pp^N$ of dimension $N - k - 1$, such that $V \cap X = \emptyset$ and $X$ is hyperbolic with respect to $V$.
\end{cor}

\begin{proof}
Suppose towards a contradiction that $X$ is hyperbolic with respect to some real linear subspace $V \subset  \pp^d$ of dimension $N - 1 - k$. Then the projection from $V$ is a real fibered morphism $f \colon X \to \pp^k$ of degree $m^k$. On the other hand, we have that $X(\R)$ is (homeomorphic to) $\pp^k(\R)$. Using the notation of the previous corollary, this means that $s=0$ and $r=1$ and therefore $\deg f=2s+r=1$ by the statement of the preceding corollary. Since $k,m \geq 2$, this is a contradiction.
\end{proof}

Two bivariate homogeneous forms with real coefficients of the same degree that interlace have the property that every nonzero polynomial in their span over $\R$ has only real zeros. One can ask whether such a phenomenon exists for forms in more than two variables. The next corollary shows that this is not possible. It might be known to experts on zero-dimensional systems, however, we have not found any such result in the literature.

\begin{cor}
 There are no $d+1$ homogeneous forms $f_0,\ldots, f_d \in \R[x_0,\ldots,x_d]$ of degree $m$ for $d,m>1$, without common zeros, such that every $d$ linearly independent forms in their span over $\R$ have just real common zeros.
\end{cor}

\begin{proof}
 If there were such polynomials $f_0,\ldots, f_d \in \R[x_0,\ldots,x_d]$, then the morphism $f:\pp^d\to\pp^d,\,x\mapsto (f_0(x):\cdots:f_d(x))$ would be real fibered. Since $f$ can be written as the composition of the Veronese embedding with a linear projection, this would imply that $\cV_{m}(\pp^d)$ is hyperbolic which is not possible by Corollary \ref{cor:veronese_not_hyperbolic}.
\end{proof}

\begin{example} \label{exp:doublecover}
 In this example, we will consider a double cover of the projective plane branched along a smooth curve of degree $2m$ without real points.
 We will show that if $m \geq 2$ this is a weakly hyperbolic variety that does not admit a hyperbolic embedding into some projective space.
 
 Let $p \in \mathbb{R}[x_0,x_1,x_2]$ be a positive definite, homogeneous polynomial of degree $2m$, such that the curve in $\mathbb{P}^2$ defined by $p$ is smooth.
Let $X$ be the hypersurface defined by $ y^2=p(x_0,x_1,x_2)$
in the weighted projective space $\mathbb{P}(1,1,1,m)$ where 
$x_0,x_1,x_2$ are homogeneous coordinates of weight $1$ and $y$ is a homogeneous coordinate of weight $m$.
We have that $X$ is a smooth projective variety
and the projection morphism $f: X \to \mathbb{P}^2$ on the first three coordinates is finite 
of degree two.
Moreover, it is real fibered since $p$ is positive definite.
If there was any hyperbolic embedding $\iota$ of $X$ into some projective space, there would be a real fibered linear projection
from $\iota(X)$ to $\pp^2$. By Corollary \ref{cor:real_point_structure} this would also have to be of degree two.
This means that $X$ would be isomorphic to a smooth quadric surface $Q$ in $\mathbb{P}^3$.
For example, comparing Hodge numbers shows that this cannot be, since we have $\textrm{h}^{1,1}(Q)=1$, but $\textrm{h}^{1,1}(X) \neq 1$ for $m \geq 2$
(see for example \cite[Chapter 17]{ara12}).
\end{example}


The previous example shows the existence of a weakly hyperbolic surface which cannot be embedded hyperbolically to some $\pp^n$ but rather to some weighted projective space. In the following, we show that this can always be done, when the real fibered morphism is flat.

\begin{lem}
Let $Y = \pp(\underbrace{1,\ldots,1}_{k+1},\underbrace{n,\ldots,n}_{d-k})$ be the weighted projective space. Let us write $y = (y_0,y_1,\ldots,y_d)$ and set $V = \left\{y \in Y \mid y_0 = y_1 = \cdots = y_k = 0\right\}$. Consider the projection from $V$, namely $f \colon Y \smallsetminus V \to \pp^k$. This map realizes $Y \smallsetminus V$ as the total space of $\cO_{\pp^k}(n)^{\oplus (d-k)}$
\end{lem}
\begin{proof}
Clearly, we have that $Y \smallsetminus V$ is a total space of a vector bundle over $\pp^k$ since over each point we perform coordinate-wise addition in the coordinates of weight $n$. Consider the distinguished affine open subsets of $\pp^k$ given by $U_j = \{x_j \neq 0\}$, for $j= 0,\ldots,k$. Note that $f^{-1}(U_j) = \{x_j \neq 0\}$ with coordinates \[(x_0/x_j,\ldots,x_{j-1}/x_j,x_{j+1}/x_j,\ldots,x_{k+1}/x_j^n,\ldots,x_d/x_j^n).\] Now on the intersection of $U_j$ with $U_i$ we get that the transition maps are diagonal with the coordinate to power $n$ on the diagonal and this corresponds precisely to $\cO_{\pp^k}(n)^{\oplus(d-k)}$.
\end{proof}

\begin{thm}
If $X$ is a real projective $k$-dimensional weakly hyperbolic variety, that admits a flat real fibered morphism $f \colon X \to \pp^k$, then we can embed $X$ into a weighted projective space $Y = \pp(\underbrace{1,\ldots,1}_{k+1},\underbrace{n,\ldots,n}_{m})$ for some $m \in \N$, such that the following diagram commutes:
\begin{equation} \label{eq:embed_into_weighted}
\xymatrix{ X \ar[r]^{\iota}\ar[dr]_f & Y \ar[d]^{\pi} \\ & \pp^k }.
\end{equation}
Here $\pi$ is the projection on the first $k+1$ coordinates.
\end{thm}
\begin{proof}
Since the morphism $f$ is finite and flat we know that $f_* \cO_X$ is a vector bundle on $\pp^k$. Furthermore, we have the trace morphism $f_* \cO_X \to \cO_{\pp^k}$ and we obtain a decomposition $f_*\cO_X \cong \cO_{\pp^k} \oplus \cE$, where $\cE$ is some vector bundle. Let $n$ be a positive integer, such that $\cE(n)$ is generated by global sections, i.e., there exits an epimorphism $\cO_{\pp^k}^{\oplus m} \to \cE(n)$ and therefore an epimorphism:
\[
\cO_{\pp^k}(-n)^{\oplus m} \oplus \cO_{\pp^k} \to f_* \cO_x.
\]
This epimorphism extends to an epimorphism $\iota ^{\#} \colon \operatorname{Sym}(\cO_{\pp^k}(-n)^{\oplus m}) \to f_* \cO_X$. Applying the relative spec construction (see \cite[Sec.\ I.3.3]{EH00}) to the sheaf of algebras  $\operatorname{Sym}(\cO_{\pp^k}(-n)^{\oplus m})$ we get the total space of the vector bundle $\cO_{\pp^k}(n)^{\oplus m}$, that we shall denote by $Z$. The epimorphism $\iota^{\#}$ induces the closed embedding $\iota \colon X \to Z$ and by construction, we have the commutative diagram \eqref{eq:embed_into_weighted}. Now it remains to apply the previous lemma to obtain that $Z$ is the complement of a linear subspace in the weighted projective space $Y$ and $\pi$ is the associated projection.
\end{proof}

\section{Some Commutative Algebra}\label{sec:comm}
In this section, we will recall definitions and theorems from commutative algebra that we will need later on. We always let $S=\mathbb{K}[z_0,\ldots,z_d]$ be the standard graded polynomial ring, where $\mathbb{K} = \C$ or $\mathbb{K} = \R$. The \textit{multiplicity} $e(M)$ of a positive-dimensional graded $S$-module $M$ is the normalized leading coefficient of its Hilbert polynomial \cite[Def. 4.1.5]{BrHr93}. The \textit{degree} $\deg\cF$ of a coherent sheaf $\cF$ on $\pp^d=\pp^d_{\mathbb{K}}$ is the multiplicity of the $S$-module $H^0_*(\cF)$.

\begin{dfn}
 An \textit{Ulrich module} is a graded Cohen--Macaulay $S$-module $M$ which is finitely generated in degree zero with the property that the minimal number of generators of $M$ is equal to its multiplicity.
\end{dfn}

Ulrich modules can be characterized in terms of their free resolutions.

\begin{thm}[Brennan, Herzog, Ulrich]\label{thm:ulrichlinear}
 Let $M$ be a finitely generated graded $S$-module. Then the following are equivalent:
 \begin{enumerate}[(i)]
  \item $M$ is an Ulrich module.
  \item The minimal $S$-free resolution 
  \[\textbf{F}: \, 0 \to F_{n-k} \to \cdots \to F_1 \to F_0 \to M \to 0\]
  of $M$ is linear, i.e., $0 \neq F_i$ is generated in degree $i$ for every $0\leq i \leq n-k$, and $k+1=\dim M$.
 \end{enumerate}
 
 If $(i)$ and $(ii)$ hold, then the rank of $F_i$ is $\binom{n-k}{i} \cdot e(M)$.
 
\end{thm}
\begin{proof}
 The equivalence is \cite[Prop. 1.5]{BHU87}. Furthermore, it is shown in \cite{ESW03} right after the proof of Prop. 2.1 that $\textrm{rank } F_i=\binom{n-k}{i} \cdot \textrm{rank } F_0$. But since $M$ is an Ulrich module, we have $e(M)=\textrm{rank } F_0$ so the additional statement follows as well.
\end{proof}

Recall, for example from \cite[\S 1.6]{BrHr93}, that given elements $f_1, \ldots, f_r\in S$ the \textit{Koszul complex} $\textbf{K}(f)=\textbf{K}(f_1,\ldots,f_r)$ is given by
\[0 \to S \to S^r \to \wedge^2 S^r \to \cdots  \to \wedge^{i} S^r \to \wedge^{i+1} S^r \to \cdots \to \wedge^{r} S^r \to 0\]
where an element $a \in \wedge^{i} S^r$ is sent to $f \wedge a \in \wedge^{i+1} S^r$ where $f=(f_1,\ldots,f_r)$.
The Koszul complex is self-dual, i.e., $\textbf{K}(f)$ and  $\Hom_S(\textbf{K}(f),S)$ are isomorphic complexes \cite[Prop. 1.6.10]{BrHr93}. The following related concept will be crucial.
\begin{dfn}\label{def:koszulmatrices}
 Let $A=(A_1, \ldots, A_s)$ be a tuple of pairwise commuting $r \times r$ matrices whose entries are in $S$. We can define on $S^r$ the structure of an
 $S[t_1,\ldots,t_s]$-module (where the $t_i$ are new variables) by letting $t_i$ act on $S^r$ via multiplication with the matrix $A_i$ from the left.
 We denote this $S[t_1,\ldots,t_s]$-module by $P$. Letting $t=(t_1,\ldots,t_s)$ we can consider the Koszul complex $\textbf{K}(t)$.
 We can consider the complex $P \otimes \textbf{K}(t)$ as a complex of $S$-modules instead of $S[t_1,\ldots,t_s]$-modules. 
 This complex is called \textit{the Koszul complex associated to the matrices} $A_1, \ldots, A_s$
 and is denoted by $\textbf{K}(A)$.\label{not:koszulmatrices}
 It is a complex of free $S$-modules and the maps of $\textbf{K}(A)$ are obtained from the maps of $\textbf{K}(t)$ by replacing everywhere $t_i$ by $A_i$.
\end{dfn}

\begin{remark}\label{rem:koszulexact}
 In the situation of the above definition, if $t_1,\ldots,t_s$ is a $P$-regular sequence, then the complex $\textbf{K}(A)$ is a free resolution of the cokernel of the matrix $(A_1 \cdots A_s)$ obtained from concatenating the matrices $A_i$. This follows from \cite[Cor. 1.6.14]{BrHr93}.
\end{remark}

We end this section with recalling a construction that we found in \cite[p. 542]{ESW03}. Consider a complex of free modules over $S$: \[\textbf{F}: 0 \to F_c \to \cdots \to F_1 \to F_0 \to 0,\] and assume that it is linear in the sense that $F_i$ is generated in degree $i$ for all $0 \leq i \leq c$. We fix a basis of each $F_i$ and 
consider the representing matrices $A_i$ (with respect to these fixed bases) of the maps $\psi_i: F_i\to F_{i-1}$ for $1\leq i \leq c$. By assumption, the entries of the matrices $A_i$ are of degree one, i.e., linear forms on $\mathbb{K}^{d+1}$. We consider these entries as degree one elements of the tensor algebra $\textnormal{T}(\mathbb{K}^{d+1})^{\vee}$. Thus the matrices $A_i$ are matrices
over the algebra $\textnormal{T}(\mathbb{K}^{d+1})^{\vee}$ and as such we can look at the product 
$\gamma(\textbf{F})=A_1 \cdots A_c$ which is a matrix whose entries
are elements of $\textnormal{T}(\mathbb{K}^{d+1})^{\vee}$ of degree $c$, i.e., multilinear forms on $\mathbb{K}^{d+1}$. Because $\textbf{F}$ is a complex,
i.e., $\psi_i \circ \psi_{i+1}=0$, these multilinear forms are in fact alternating multilinear forms. 
Thus the entries of the matrix $\gamma(\textbf{F})$ are elements of $\bigwedge^{c} (\mathbb{K}^{d+1})^{\vee}$. Up to multiplication of $\gamma(\textbf{F})$
from the left by a matrix from $\GL_{\rank{F_0}}(\mathbb{K})$ and from the right by a matrix from $\GL_{\rank{F_c}}(\mathbb{K})$. This does not depend on the
choice of the bases of the $F_i$. We call the matrix $\gamma(\textbf{F})$
 constructed above the \textit{alternating matrix} associated to $\textbf{F}$.

\begin{remark}
 In the following, we will not specify the bases of $F_0$ and $F_c$ if it is either clear which bases we choose or if the properties of $\gamma(\textbf{F})$ that we are interested in are independent of such a choice.
\end{remark}

\begin{example}\label{exp:alternatingmatrix}
 Let $l=(l_1, \ldots, l_r)$ where the $l_i \in S$ are homogeneous elements of degree one.
 Then the Koszul complex $\textbf{K}(l)$ is a linear complex. The alternating matrix $\gamma(\textbf{K}(l))$ has the size of $1 \times 1$.
 Its entry is the alternating form on $r$ copies of $\mathbb{K}^{d+1}$ that sends $(v_1, \ldots, v_r)$ to the determinant of the $r \times r$ matrix
 $(l_i(v_j))_{1\leq i,j \leq r}$.
 
 More generally, let $A_1, \ldots, A_r$ be matrices with linear entries from $S$ that commute pairwise
 and let $\textbf{K}(A)$ be the Koszul complex associated to the $A_i$  (see Definition \ref{def:koszulmatrices}).
 This is a linear complex.
 The alternating matrix $\gamma(\textbf{K}(A))$ is the alternating form on $r$ copies of $\mathbb{K}^{d+1}$ that sends $(v_1, \ldots, v_r)$ to
 the matrix
 \[
  \gamma(\textbf{K}(A))(v_1,\ldots,v_r)=\sum_{\sigma \in \mathfrak{S}_r} \textnormal{sgn}( \sigma) \cdot A_{\sigma(1)}(v_1)  \cdots A_{\sigma(r)}(v_r),
 \]
 where $A_i(v_j)$ is supposed to be the matrix whose entries are the entries of $A_i$ evaluated at $v_j$ and $\mathfrak{S}_r$ is the symmetric group on $r$ elements.
 Later on, we will be interested in the case where the matrices $A_i$ are symmetric. Then $\gamma(\textbf{K}(A))$ is also symmetric.
 Indeed, $\gamma(\textbf{K}(A))^{t}$ sends  $(v_1, \ldots, v_r)$ to
 the matrix
 \[
  \gamma(\textbf{K}(A))^{t}(v_1,\ldots,v_r)=\sum_{\sigma \in \mathfrak{S}_r} \textnormal{sgn}( \sigma) \cdot A_{\sigma(r)}(v_r)^{t} 
  \cdots A_{\sigma(1)}(v_1)^{t}.
 \]Letting $\tau$ be the permutation that maps $\tau(j)=r-j+1$ for all $j=1,\ldots,r$ and because the $A_i$ are symmetric, this matrix equals
  \[
  \sum_{\sigma \in \mathfrak{S}_r} \textnormal{sgn}( \sigma)\cdot \textnormal{sgn}( \tau) \cdot A_{\sigma(1)}(v_r)  \cdots A_{\sigma(r)}(v_1).
 \]
 Thus we have \[
               \gamma(\textbf{K}(A))^{t}(v_1,\ldots,v_r)=\textnormal{sgn}( \tau) \gamma(\textbf{K}(A))(v_r,\ldots,v_1)
               = \gamma(\textbf{K}(A))(v_1,\ldots,v_r).
              \]
 The last equality holds because the entries of $\gamma(\textbf{K}(A))$ are alternating forms.
\end{example}

\section{Admissible Determinantal Representations and Ulrich Sheaves} \label{sec:ulrich}

In this section, we will work over the complex numbers $\C$ unless explicitly otherwise stated. We let $S=\mathbb{C}[z_0,\ldots,z_d]$.

\begin{dfn}
  A coherent sheaf $\cF$ on $\pp^d$ is an \textit{Ulrich sheaf} if the $S$-module $H^0_*(\cF)$ is an Ulrich module.
\end{dfn}

Now let $X$ be a projective variety of pure dimension $k$ and fix an embedding $i \colon X \to \pp^d$ given by a line bundle $\cL = i^*\cO_{\pp^d}(1)$, we say that a sheaf $\cF$ on $X$ is Ulrich with respect to $\cL$ if $i_* \cF$ is Ulrich. In this case, we also have that if we decompose into irreducible components $X = X_1 \cup \cdots \cup X_r$, then by \cite[VI \S2, Prop. 2.7]{kollar} we have that $\deg \cF = \sum_{j=1}^r \rank (\cF|_{X_j}) \deg(X_j)$, where the degree of each component is with respect to the embedding $i$, see also \cite{ESW03}. When the embedding is fixed and there exists such a sheaf on $X$ of degree $n$, then we will simply say that $X$ admits an Ulrich sheaf of degree $n$.

There is yet another equivalent way to define Ulrich sheaves. Given a subscheme $X \subset  \pp^d$ of pure dimension $k$, we can realize $X$ as a branched covering of $\pp^k$ by means of a linear projection from a linear subspace of $\pp^d$ of dimension $d -k -1$ disjoint from $X$. One can define then a sheaf $\cF$ supported on $X$ (scheme-theoretically) to be Ulrich if for a general linear projection $\pi \colon X \to \pp^k$, there exists a positive integer $m$, such that $\pi_* \cF \cong \cO_{\pp^k}^m$. See \cite[Prop.\ 2.1]{ESW03} for the equivalence of those definitions.

We now recall the definition of determinantal representations of subvarieties of $\pp^d=\pp^d_\C$ introduced in \cite{SV14}.

\begin{dfn} \label{dfn:Livsic-type_det_rep}
We say that $X \subset  \pp^d$ of dimension $k$ has a Livsic-type determinantal representation, if there exists a tensor $\gamma \in \wedge^{k+1} \C^{d+1} \otimes \textnormal{M}_n(\C)$, such that the set of closed points $p \in \pp^d$, satisfying $\gamma \wedge p$ has non-trivial kernel considered as a linear map from $\C^n$ to $\wedge^{k+2}\C^{d+1} \otimes \C^n$, is precisely the set of closed points of $X$.
\end{dfn}

Consider the kernel sheaf $\cK$ of the vector bundle map $\cO_{\pp^d}(-1)^n \to \cO_{\pp^d}^{\binom{d+1}{  k+2} n }$ associated to $\gamma$. Let us associate a cohomology cycle to $\gamma$. We decompose $X = X_1 \cup \cdots \cup X_r$ into irreducible components, and for each component, we set $n_j$ to be the dimension of the fiber of $\cK$ at the generic point of $X_j$. We then define the cycle of $\gamma$ to be:
\[
Z(\gamma) = \sum_{j=1}^r n_j [X_j].
\]
Let us denote (here $[H]$ is the class of a hyperplane):
\[
\deg(\gamma) = \deg(Z(\gamma)) = \int_{\pp^d} Z(\gamma) \cdot [H]^k =  \sum_{\dim X_j = k} n_j \deg(X_j).
\]

\begin{dfn} \label{dfn:admissible}
We say that $X$ has an admissible (very reasonable in the parlance of \cite{SV14}) determinantal representation if $X$ has a determinantal representation $\gamma \in \wedge^{k+1} \C^{d+1} \otimes \textrm{M}_n(\C)$ and $\deg(\gamma) = n$. 
\end{dfn}


\begin{remark}\label{rem:chow}
Let $\gamma \in \wedge^{k+1} \C^{d+1} \otimes \textrm{M}_n(\C)$ be an admissible determinantal representation of the projective variety $X\subset \pp^d$ of dimension $k$. Then the Chow form of $X$ has a determinantal representation. Indeed, we have
 \[
  \wedge^{k+1} \C^{d+1} \otimes \textrm{M}_n(\C) \cong (\wedge^{d-k} \C^{d+1})^\vee \otimes \textrm{M}_n(\C)
 \]
 and we can think of $\gamma$ as a matrix having linear forms on $\wedge^{d-k} \C^{d+1}$ as entries. Let $Y\subset \pp(\wedge^{d-k} \C^{d+1})$ denote the corresponding determinantal hypersurface defined by $\det \gamma$. Consider the Pl\"ucker embedding of the Grassmannian $\G(d-k-1,d)\subset \pp(\wedge^{d-k} \C^{d+1})$. The intersection $Y\cap\G(d-k-1,d)$ consists of exactly those linear subspaces that intersect $X$ \cite[Thm.\ 2.18]{SV14}. In particular, the admissible determinantal representation $\gamma$ gives a linear determinantal representation of some power of the \textit{Chow form} of $X$.
\end{remark}

\begin{ex}
It is, however, not true that a determinantal representation of the Chow form of $X$ yields a determinantal representation of $X$ itself. To see this let $\delta_i$ be the standard basis for $\C^3$ and set $X = \left\{[\delta_0], [\delta_1], [\delta_2] \right\} \subset \pp^2$, where the square brackets stand for the projective equivalence class of the point. Let us write $x_{01}, x_{02}, x_{12}$ for the coordinates of the dual projective space, where for example $x_{01}$ is the coordinate vanishing on the hyperplane orthogonal to $\delta_2$. Note that the Chow form of $X$ is the union of the coordinate hyperplanes and the following matrix is a determinantal representation for the Chow form
\[
\begin{pmatrix}
x_{01} & 0 & 0 \\ 0 & x_{02} & x_{01} \\ 0 & 0 & x_{12}
\end{pmatrix}.
\]
The conressponding tensor is:
\[
\gamma = \begin{pmatrix} \delta_2 & 0 & 0 \\ 0 & -\delta_1 & \delta_2  \\ 0 & 0 & \delta_0 \end{pmatrix} \in \C^3 \otimes M_3(\C).
\]
Now note that 
\[
\gamma \wedge \delta_0 = \begin{pmatrix} - \delta_0 \wedge \delta_2 & 0 & 0 \\ 0 & \delta_0 \wedge \delta_1 & -\delta_0 \wedge \delta_2 \\ 0 & 0 & 0 \end{pmatrix}.
\]
This is an injective map from $\C^3$ to $\wedge^2 \C^3 \otimes \C^3$ and thus $\gamma$ is not a determinantal representation of $X$, since $[\delta_0] \in X$.
\end{ex}

\begin{lem} \label{lem:reduction_to_commuting}
Let $X \subset \pp^d$ be a subvariety of dimension $k$ and let $\gamma \in \wedge^{k+1} \C^{d + 1} \otimes M_n(\C)$ be an admissible determinantal representations for $X$. Then there exist commuting matrices of linear forms $T_0, \ldots, T_{d-k-1}$, such that $X$ is precisely the collection of points, where the long matrix $T= \left(T_0,\ldots, T_{d-k-1}\right)$ has a left kernel. Furthermore, for every $0 \leq j \leq d-k-1$, $T_j = z_j I - A_j$, where $A_j$ is a matrix of linear forms in the variables $z_{d-k}, \ldots, z_d$ and the $A_j$ are generically semi-simple.
\end{lem}

\begin{proof}
We fix a subspace $V$ off the Chow form of $X$ and a basis $e_0,\ldots,e_{d-k-1}$ for $V$. We complete our basis to a basis of $\C^{d+1}$ that we will denote by $e_0,\ldots,e_d$. As we have seen in \ref{rem:chow} for any $d-k-1$-dimensional subspace $V \subset \pp^d$ that is off the Chow form of $V$, the matrix $\gamma \wedge (e_0 \wedge \cdots e_{d-k-1})$ is invertible.  For every point $u = \sum_{j=0}^d z_j e_j \notin V$ and every $0 \leq i \leq d-k-1$, we define the matrices
\begin{multline*}
\gamma(V,i,u) = \gamma \wedge (e_0 \wedge \cdots \wedge e_{i-1} \wedge u \wedge e_{i+1} \wedge \cdots \wedge e_{d-k-1}) = \\ z_i \gamma(V) + (-1)^{d-k-i -1} \sum_{j=d-k}^d z_j \gamma \wedge (e_0 \wedge \cdots \wedge e_{i-1} \wedge e_{i+1} \wedge \cdots \wedge e_{d-k-1} \wedge e_j).
\end{multline*}
We set $T_i^t = \gamma(V)^{-1} \gamma(V,u,i)$, these are matrices of linear forms in the variables $z_0,\ldots,z_d$. By \cite[Cor. 2.21]{SV14} these matrices pairwise commute. By \cite[Cor. 6.3.18]{kummerdiss} the variety $X$ is precisely the collection of points where the $T_i^t$ have joint kernel. 

Note that $T_i = z_i I - A_i$, for $0 \leq i \leq d-k-1$, where $A_i = (-1)^{d-k-i} \sum_{j=d-k}^d z_j (\gamma \wedge (e_0 \wedge \cdots \wedge e_{i-1} \wedge e_{i+1} \wedge \cdots \wedge e_{d-k-1} \wedge e_j))^t$. Thus $A_i$ are matrices of linear forms in the variables $z_{d-k},\ldots,z_d$. Furthermore, since $\gamma$ is admissible, we have that for a generic point $v' = \sum_{j=d-k}^d z_j e_j$, the joint kernels of the $T_i$ at the points $X \cap \operatorname{Span}\{V,v'\}$ span $\C^{d+1}$. These are precisely the joint eigenspaces of the $A_i$ and thus the $A_i$ are generically semi-simple.
\end{proof}

The next theorem shows that $X$ has an admissible determinantal representation if and only if $X$ admits an Ulrich sheaf. Note that Eisenbud and Schreyer proved in \cite[Thm.\ 0.3]{ESW03} that if a variety admits an Ulrich sheaf, then the Chow form of the variety has a determinantal representation. The next theorem completes the picture described there by characterizing those determinantal representations of the Chow form that arise from Ulrich sheaves. One direction in the following proof is a refinement of the proof of \cite[Thm.\ 0.3]{ESW03}.
%
%

\begin{thm} \label{thm:ulrich_iff_admissible}
The following conditions for a subvariety $X \subset  \pp^d$ of dimension $k$, are equivalent:
\begin{enumerate}[(i)]
 \item $X$ admits an Ulrich sheaf of degree $n$.
 \item  $X$ has an admissible determinantal representation $\gamma$ such that $\deg(\gamma) = n$.
\end{enumerate}
\end{thm}
\begin{proof}
Assume that there exists an Ulrich sheaf $\cF$ supported on $X$. Denote the module of twisted global sections of $\cF$ by $M = \oplus_{j \in \N} \textrm{H}^0(\pp^d,\cF(j))$. Consider the linear free resolution of $M$:
\begin{equation} \label{eq:resolution}
\textbf{F}:\xymatrix{ 0 \ar[r] & F_{d - k} \ar[r]^{\psi_{d-k}} & \cdots \ar[r]^{\psi_2} & F_1 \ar[r]^{\psi_1} & F_0 \ar[r] & M \ar[r] & 0}.
\end{equation}
We want to show that the transpose $\gamma(\textbf{F})^{t}$ of the associated alternating matrix is an admissible determinantal representation of $X$.
For $v \in \C^{d+1}$ the matrix $\gamma(\textbf{F}) \wedge v$ is (up to a sign) the same as the matrix that we obtain by plugging in the $v_i$ for the $z_i$ in one of the $\psi_j$ (for example $\psi_1$) before composing them. This is because if we regard $\gamma(\textbf{F})$ and $\gamma(\textbf{F}) \wedge v$ as matrix-valued alternating multilinear forms as in Section \ref{sec:comm}, we have $(\gamma(\textbf{F}) \wedge v)(v_2,\ldots,v_{d-k})=\gamma(\textbf{F})(v,v_2,\ldots,v_{d-k})$ for all $v_i\in\C^{d+1}$.
Thus since $X$ is the support of $\cF$ we see that $\gamma(\textbf{F}) \wedge v$ has a nontrivial left kernel if and only if $[v] \in X$.
The left kernel of $\gamma(\textbf{F}) \wedge v$ for a general point from an irreducible component $X_j$ of $X$ is exactly $\rank (\cF|_{X_j})$.
Thus it follows from \[n=\deg \cF = \sum_{j=1}^r \rank (\cF|_{X_j}) \deg(X_j)\] that $\gamma(\textbf{F})^{t}$ is an admissible determinantal representation of Livsic-type.
Let $T = \left(T_0, \ldots, T_{d-k-1} \right)$ be the long matrix obtained in Lemma \ref{lem:reduction_to_commuting}. If we plug in a point $p\in\pp^d$, then $T$ has nontrivial left kernel if and only if $p \in X$. Thus the reduced support of the sheaf $\cF=\widetilde{M}$ where $M= \coker (T)$ is precisely $X$. 
Recall that $T_j = z_j I - A_j$, where $A_j$ is a matrix of linear forms in the variables $z_{d-k},\ldots, z_d$. We let $R=\C[z_{d-k},\ldots,z_d]$ and let $z_{i}$ act on $R^n$ via $A_{i}$ for  $i=0,\ldots,{d-k-1}$. This gives a graded $S$-module which is isomorphic to $M$. It is generated in degree zero with the minimal number of generators being equal to $n$. On the other hand, the Hilbert function of this $S$-module is the same as the Hilbert function of the $R$-module $R^n$. Therefore $e(M)=n$ and $M$ is an Ulrich module. Because the matrices $A_i$ are simultaneously diagonalizable at a general point of $\C^{k+1}$, we see that the annihilator of $M$ is a radical ideal. Therefore, $\cF$ is an Ulrich sheaf of degree $n$ with scheme theoretic support $X$.
%
\end{proof}

\begin{remark}\label{rem:koszulres}
 The free resolution of the Ulrich module defined in the second part of the previous proof is the Koszul complex associated to the matrices $T_0,\ldots,T_{d-k-1}$ by Remark \ref{rem:koszulexact}. Since it is linear, this is another way of seeing that it is an Ulrich module.
\end{remark}

From the proof of the preceding theorem, we also get the following statement for Ulrich sheaves on $\R$-varieties.

\begin{cor}
 Let $X \subset  \pp^d_\R$ be a projective $\R$-variety. The following are equivalent:
\begin{itemize}
\item[(i)] There exists an Ulrich sheaf $\cF$ on $\pp^d_\R$ supported on $X$,

\item[(ii)] There is a  tensor $\gamma \in \wedge^{k+1} \R^{d+1} \otimes \textnormal{M}_n(\R)$ which is an admissible determinantal
representation of $X_\C$.
\end{itemize}
\end{cor}

\begin{example}
 It was shown in \cite{BRW05} that the variety of $m \times n$ matrices of rank at most $r$ admits a rank one Ulrich sheaf for all $1 \leq r \leq \min\{m,n\}$.
 Thus determinantal varieties have determinantal representations.
\end{example}

Let $X \subset  \pp^d$ be a subvariety of pure dimension $k$ and assume that $X$ has an admissible determinantal representation of size $n$. Note that the group $\GL_n(\C)$ acts both from the left and from the right on the set of admissible determinantal representation of size $n$. The actions are induced from the natural actions of $\GL_n(\C)$ on $\wedge^{k+1} \C^{d+1} \otimes M_n(\C)$ by left and right multiplication on the second coordinate of the tensor product.  We will say that two admissible tensors $\gamma_1$ and $\gamma_2$ are similar, if there exist matrices $A,B \in \GL_n(\C)$, such that $\gamma_1 = A \gamma_2 B$. Then we have the following result:

\begin{prop} \label{prop:uniqueness}
The association of isomorphism classes of Ulrich sheaves and similarity classes of determinantal representations described in Theorem \ref{thm:ulrich_iff_admissible} is a bijection.
\end{prop}
\begin{proof}
Assume that we have an isomorphism $\varphi \colon \cF_1 \to \cF_2$. Let $M_j$ be the module of twisted global sections of $\cF_j$. Then $\varphi$ induces an isomorphism $M_1 \cong M_2$, that can be lifted to an isomorphism of their minimal free resolutions (see \cite[Thm.\ 1.6]{Eis05})
$$
\xymatrix{ 0 \ar[r] & S(-d - k)^n \ar[r] \ar[d]^{A_{d-k}} & \cdots \ar[r] & S(-1)^{n (d-k)} \ar[r]^{\quad\quad T} \ar[d]_{A_1} & S^n \ar[r] \ar[d]^{A_0} & M_1 \ar[r] \ar[d]^{\varphi} & 0 \\ 0 \ar[r] & S(-d - k)^n \ar[r] & \cdots \ar[r] & S(-1)^{n (d-k)} \ar[r]^{\quad\quad T} & S^n \ar[r] & M_2 \ar[r] & 0}.
$$
Note that each $A_j$ is an invertible complex matrix. Hence we have that $\psi_{1j} = A_{j-1}^{-1} \psi_{2j} A_j$. Now note that the constructions of Theorem \ref{thm:ulrich_iff_admissible} are inverse to each other and that the admissible determinantal representations are determined by the commuting matrices from Lemma \ref{lem:reduction_to_commuting}.
\end{proof}

The following corollary is a strengthening of \cite[Thm.\ 6.2]{SV14}.

\begin{cor} \label{cor:all_det_reps_curves}
Every projective curve $X \subset  \pp^d$ has an admissible determinantal representation of size $\deg X$ and in the case of smooth irreducible curves the algorithm of \cite[Thm.\ 6.2]{SV14} constructs them all.
\end{cor}
\begin{proof}
By \cite[Prop.\ 4.4]{ESW03} every projective curve $X \subset  \pp^d$ admits an Ulrich sheaf of rank $1$. Therefore it has an admissible determinantal representation of size $\deg X$ by Theorem \ref{thm:ulrich_iff_admissible}.

Now let $X \subset  \pp^d$ be smooth and irreducible. By \cite[Prop.\ 4.4]{ESW03} and Proposition \ref{prop:uniqueness} there is a one-to-one correspondence between non-special line bundles of degree $g-1$ on $X$ and admissible determinantal representations of $X$ of size $\deg X$. Recall that a line bundle $\cL$ on $X$ of degree $g-1$ is said to be non-special if $h^0(\cL) = h^1(\cL) = 0$. The algorithm provided in \cite[Thm.\ 6.2]{SV14} constructs out of a non-special line bundle of degree $g-1$ on $X$ an admissible determinantal representation of $X$ of degree $\deg X$. On the other hand, we can recover the Ulrich line bundle from the algorithm as follows. Note that by \cite[Cor.\ 4.9]{SV14} we can find the right (and similarly the left) kernel of the determinantal representation. The left kernel is isomorphic to $\cL \otimes \iota^* \cO_{\pp^d}(1)$, where $\cL$ is the non-special line bundle that we have started with and $\iota$ is the embedding of $X$ into $\pp^d$. To see this note that the fiber of the left kernel at a point $p \in X \setminus \operatorname{Supp}(D)$ is spanned by $u^{\times}_{D,\ell}(p)$ (see \cite[pp.\ 18]{SV14} for the definition), where $D$ is the divisor of $\iota^* \cO_{\pp^d}(1)$. Note that since the curve is smooth and the degree of the determinantal representation is $\deg X$, the left kernel is a line bundle and in fact $u^{\times}_{D,\ell}$ is a section of the left kernel. We conclude that the left kernel is isomorphic to $\cL \otimes \iota^* \cO_{\pp^d}(1)$ (the reader can also compare this argument with \cite{Vin89} for the case of plane curves). Thus $\cL \otimes \iota^* \cO_{\pp^d}(1)$ is precisely the Ulrich sheaf of degree $\deg X$ supported on $X$ that corresponds to the determinantal representation in the proof of Theorem \ref{thm:ulrich_iff_admissible}.
\end{proof}

\section{Real Varieties and Bilinear Forms} \label{sec:real_bilinear}
In this section, we work again over the ground field $\R$. Again we write $\pp^d=\pp^d_\R$.
Let $X$ be an $\R$-variety
and let $\cE$ be a coherent sheaf on $X$.
A non-degenerate $\cE$-valued bilinear form on a coherent sheaf $\cF$ is a map $\varphi \colon \cF \otimes \cF \to \cE$, 
such that the adjoint morphism $\kappa \colon \cF \to \shHom_{\cO_X}(\cF,\cE)$ is an isomorphism. We will call a form symmetric if $\varphi = \varphi \circ \epsilon$, where $\epsilon \colon \cF \otimes \cF \to \cF \otimes \cF$ is the isomorphism induced
 by the morphism of presheaves sending a section $f \otimes g$ to $g \otimes f$. If $\cF$ is reflexive and $\cE$ is a line bundle, then we have that $\varphi$ is symmetric if and only if we have that $\kappa^t = \kappa$, where $\kappa$ is the morphism obtained by applying $\shHom_{\cO_X}(-,\cE)$ to $\kappa$.



\begin{dfn} \label{dfn:positive_form}
Let $\cF$ be a coherent sheaf on an $\R$-variety $X$ and assume that it admits a non-degenerate symmetric $\cO_X$-valued bilinear form $\varphi$. Then we say that $\varphi$ is positive (resp. negative) definite if for every closed point $x \in X(\R)$ we have that the induced form on the fiber of $\cF$ is positive (resp. negative) definite.

\end{dfn}

\begin{rem}
If $\cF \cong \cO_X^n$, then a $\cO_X$-valued bilinear form is positive (resp. negative) definite if and only if the induced form on the global sections is positive (resp. negative) definite.
\end{rem}

\begin{rem}
More generally one can define definite $\cL$-valued bilinear forms on $\cF$, for $\cL$ a line bundle on $\cF$, following \cite[\S II.7.2]{Knu91}. Namely, we will say that a non-degenerate bilinear $\cL$-valued form on $\cF$ is (semi-)definite if for every closed point $x \in X(\R)$ 
we have that the induced form on the fiber of $\cF$ is (semi-)definite. In this case, however, one cannot speak of positive or negative definite forms, since this depends on the chosen trivialization.
\end{rem}

\begin{dfn} \label{dfn:Livsic-type_det_rep_sym}
Let $X \subset  \pp^d$ be a projective $\R$-variety of dimension $k$. Let $\gamma \in \wedge^{k+1} \R^{d+1} \otimes \textrm{M}_n(\R)$ be a Livsic-type determinantal representation of $X_\C$. Having a basis $e_0, \ldots, e_d$ of $\R^{d+1}$ and $I=\{i_0,\ldots,i_k\} \subset  \{0, \ldots, d \}$ with $i_0<i_1<\cdots<i_k$ we denote $e_I=e_{i_0}\wedge\cdots\wedge e_{i_k}$. If for some (and hence for every) basis $e_0, \ldots, e_d$ of $\R^{d+1}$ we can write \[\gamma= \sum_{I \subset  \{0, \ldots, d \}, \, |I|=k+1} e_I \otimes\gamma_I\] for some real \textit{symmetric} $n\times n$ matrices $\gamma_I$ over $\R$, then we say that $\gamma$ is a real \textit{symmetric} Livsic-type determinantal representation.
\end{dfn}

\begin{dfn}
 Let $X \subset  \pp^d$ be a projective $\R$-variety of dimension $k$. Let $\gamma \in \wedge^{k+1} \R^{d+1} \otimes \textrm{M}_n(\R)$ be a symmetric admissible determinantal representation of $X_\C$. As explained in Remark \ref{rem:chow} we can think of $\gamma$ as a determinantal representation of some power of the Chow form of $X$. Evaluating $\gamma$ at a linear space $E\in\G(d-k-1,d)$ is well defined up to a nonzero scalar factor. In particular, it makes sense to talk about the rank and, if $E$ is a real linear space, the signature (up to sign) of $\gamma$ at $E$. We say that $\gamma$ is \textit{definite at} $E$ if $E$ is a real linear space and the signature of $\gamma$ at $E$ is $n$ (or $-n$).
\end{dfn}

We have seen in the preceding section that admissible determinantal representations of a variety $X \subset  \pp^d$ correspond to Ulrich sheaves supported on $X$. It was shown in \cite[Prop. 3.12]{SV14} that the existence of a real symmetric admissible Livsic-type determinantal representation for $X_\C$ which is definite at a linear space $V$ (of correct dimension) implies that $X$ is hyperbolic with respect to $V$. 

In the following, we elaborate what the properties of being real symmetric and definite at a certain linear space mean for the corresponding Ulrich sheaf.
We will use notations from the theory of Grothendieck duality in the case of
finite morphisms. Let $f: X \to Y$ be a finite morphism of noetherian schemes. Let $\cG$ be a quasi-coherent sheaf on $Y$ and consider the sheaf
$\shHom_{\cO_Y}(f_* \cO_X, \cG)$. Since this is a quasi-coherent $f_* \cO_X$-module, it corresponds to a quasi-coherent $\cO_X$-module
which we will denote by $f^! \cG$.\label{not:uppershriek}
An introduction to the theory of Grothendieck duality in its full generality (and not just for finite morphisms) is given in \cite{AK70} or in \cite{Con00}. We will mostly use the notations of \cite{AK70}.
For the reader not familiar with the ideas of Grothendieck duality we recall the following basic lemma \cite[III \S 6, Ex. 6.10]{Har77}.
\begin{lem}\label{lem:grothbasic}
 Let $f: X \to Y$ be a finite morphism of noetherian schemes.
 Let $\cF$ be a coherent sheaf on $X$ and $\cG$ be a quasi-coherent sheaf on $Y$. There is a natural isomorphism
 \[f_* \shHom_{\cO_X}(\cF, f^! \cG) \to \shHom_{\cO_Y}(f_* \cF, \cG)\]
 of quasi-coherent $f_*\cO_X$-modules.
\end{lem}
Let $f: X \to Y$ be a finite morphism of noetherian schemes. Let $\cF$ be a coherent sheaf on $X$ and consider an $f^! \cO_Y$-valued bilinear form on $\cF$, i.e., a morphism  $\cF \otimes\cF \to f^! \cO_Y$ of coherent $\cO_X$-modules. This corresponds to a morphism $\cF \to \shHom_{\cO_X}(\cF, f^! \cO_Y)$. Lemma \ref{lem:grothbasic} tells us that this gives us a morphism (of quasi-coherent $f_*\cO_X$-modules)
 \[
  f_* \cF \to \shHom_{\cO_Y}(f_* \cF, \cO_Y)
 \]
 which gives rise to an $\cO_Y$-valued bilinear form on the pushforward $f_* \cF$.

In the special case when $\cF$ is an Ulrich sheaf supported on a $k$-dimensional projective $\R$-variety $X\subset \pp^d$ and $f: X\to\pp^k$ is a finite surjective linear projection, a non-degenerate $f^! \cO_{\pp^k}$-valued bilinear form on $\cF$ gives rise to a non-degenerate $\cO_{\pp^k}$-valued bilinear form on $\cO_{\pp^k}^N$.

\begin{thm}
 Let $X \subset  \pp^d$ be a projective $\R$-variety of dimension $k$. Let $E \subset  \pp^d$ be a real linear subspace of dimension 
 $d-k-1$, such that
 $E \cap X = \emptyset$ and let $f: X \to \pp^k$ be the linear projection from center $E$.
 Then the following are equivalent:
 \begin{enumerate}[(i)]
  \item There exists an Ulrich sheaf $\mathcal{F}$ supported on $X$ together with a non-degenerate
   $f^! \cO_{\pp^k}$-valued symmetric bilinear form on $\cF$, such that the corresponding $\cO_{\pp^k}$-valued
   form on $\cO_{\pp^k}^N$ is positive definite.
  \item The complexification $X_\C$ has an admissible real symmetric
  determinantal representation $\gamma\in \wedge^{k+1} \R^{d+1} \otimes \textrm{M}_n(\R)$ that is definite at $E$.
 \end{enumerate}
\end{thm}

\begin{proof}
 Let $R=\textrm{H}_*^0 \cO_{\pp^k} \cong \R[z_{d-k},\ldots,z_d]$
 and let $S=\textrm{H}^0_* \cO_{\pp^d} \cong \R[z_0,\ldots,z_d]$.
 We can assume that under these identifications the linear projection $f$ corresponds to the inclusion 
 $\R[z_{d-k},\ldots,z_d] \hookrightarrow \R[z_0,\ldots,z_d].$

 First, assume that there is such an Ulrich sheaf. Let the bilinear form given by the morphism $\varphi:\cF \to \shHom_{\cO_X}(\cF, f^! \cO_Y)$ of $\cO_X$-modules.  Furthermore, let $M=\textrm{H}_*^0 \mathcal{F}$.
 As an $R$-module we have $M=R^N$ and $z_{0}, \ldots, z_{d-k-1}$
 act like matrices $A_{0},\ldots, A_{d-k-1}$ with homogeneous
 elements of degree one from $R$ as entries.
 The fact that $\varphi$ is a morphism of $\cO_X$-modules translates to these matrices being selfadjoint with respect to the corresponding $\cO_{\pp^k}$-valued
   symmetric bilinear form on $\cO_{\pp^k}^N$. Since it is positive definite, there is a basis of $R^N$ with respect to which $A_{0}, \ldots, A_{d-k-1}$ are symmetric. Let $T_i=z_{i}-A_{i}$. The free resolution of $M$ is given by the Koszul complex $\textbf{K}(T)$ associated to $T_0, \ldots, T_{d-k-1}$ by Remark \ref{rem:koszulres}.
 In Example \ref{exp:alternatingmatrix} we have seen that the matrix $\gamma(\textbf{K}(T))$ is symmetric. In the proof of Theorem \ref{thm:ulrich_iff_admissible} we have seen that $\gamma(\textbf{K}(T))$ is an admissible determinantal representation. Furthermore, the alternating form defined by $\gamma(\textbf{K}(T))$ is given by 
 \[
  \gamma(\textbf{K}(T))(v_0,\ldots,v_{d-k-1})=\sum_{\sigma \in \mathfrak{S}_{d-k}} \textnormal{sgn}( \sigma) \cdot T_{\sigma(0)}(v_0)  \cdots T_{\sigma(d-k-1)}(v_{d-k-1}).
 \]
 Letting $v_i$ be the $i$th unit vector $\delta_{i}$, the above expression is the identity matrix
 since $T_i$ evaluated at $\delta_{j}$ is the identity matrix if $i=j$ and zero otherwise for all $i,j=1, \ldots, d-k$. This shows positive definiteness.
 
 Now we assume that there is such an admissible symmetric determinantal representation $\gamma$ as in $(ii)$.
 We can assume that $\gamma$ evaluated at $E$ is the identity matrix.
 Let us define an Ulrich sheaf on $X$, let $T_0,\ldots,T_{d-k-1}$ be the matrices of linear forms of size $n \times n$ obtained from $\gamma$ as in the second part of the proof of Theorem \ref{thm:ulrich_iff_admissible}.
 We can write $T_{i}=z_{i}-A_{i}$ where $A_{i}$ is an 
 $n \times n$ matrix whose entries are linear forms in the variables $z_{d-k},\ldots,z_d$.
 The matrices $T_{i}$, and thus also the matrices $A_{i}$, commute pairwise.
 By letting $z_{i}$ act on $R^n$ via $A_{i}$ we get a graded $S$-module $M$. It follows from the proof of Theorem \ref{thm:ulrich_iff_admissible} that $M$ is an Ulrich module. Let $\cF=\widetilde{M}$ be the corresponding Ulrich sheaf. The isomorphism $M \to \Hom_R(M,R)$ that sends the standard basis of $R^n$ to its dual basis gives us the desired $f^! \cO_{\pp^k}$-valued bilinear form on $\iota^* \cF$ because the matrices $A_{i}$ are symmetric. By construction, the standard basis is an orthonormal basis and therefore the bilinear form is positive definite.
\end{proof}

\begin{rem}
  It was shown in \cite{HUB91} that if $X \subset  \pp^d$ is a complete intersection, then $X$ has an Ulrich sheaf. 
  If a variety $X \subset  \pp^d$ admits an Ulrich sheaf that satisfies the positivity conditions of the preceding theorem, then $X$ must be hyperbolic.
  But not every hyperbolic complete intersection has such an Ulrich sheaf: Br\"and\'en \cite{Bra11}
  constructed a hyperbolic hypersurface, such that no power of its defining polynomial can be written as the determinant of a real symmetric
  matrix
  with linear entries that is positive definite at some point.
  On the other hand, it was shown in \cite{Kum} that for every smooth hyperbolic hypersurface $X \subset  \pp^d$ there is a 
  hypersurface $Y \subset  \pp^d$, such that $X \cup Y$ admits an Ulrich sheaf with the positivity conditions of the above theorem.
  This is very much related to the generalized Lax conjecture, see for example \cite[Conjecture 3.3]{Vppf}.
  The following question is the natural extension of the result mentioned above to varieties of higher codimension.
\end{rem}

\begin{quest} \label{quest:generalized_lax}
 Let $X \subset  \pp^d$ be a smooth and irreducible variety of dimension $k$
 which is hyperbolic with respect to the linear space $E$. Is there a variety $Y \subset  \pp^d$ of the same dimension as $X$
 and also hyperbolic with respect to $E$
 such
 that $X \cup Y$ admits an Ulrich sheaf $\mathcal{F}$  together with a nondegenerate symmetric
   $f^! \cO_{\pp^k}$-valued bilinear form $\varphi$ on $\cF$, such that the corresponding symmetric
   form on $\cO_{\pp^k}^N$ is positive definite? Here $f: X \cup Y \to \pp^k$ denotes the linear projection from center $E$.
\end{quest}

Let $f \colon X \to Y$ be a finite flat morphism of degree $m$. Let $\cF$ be a coherent sheaf on $X$.
We say that $\cF$ is  $f$-positive (resp. $f$-nonnegative) if
it admits a non-zero $f^! \cO_Y$-valued
bilinear form, such that the induced $\cO_Y$-valued bilinear form on $f_* \cF$ is symmetric and positive definite (resp. positive semidefinite).
Following \cite{KMS14} one can define  the notion of $f$-Ulrich sheaves, namely a sheaf $\cF$ on $X$ is $f$-Ulrich
if there exists a positive integer $r$, such that $f_*\cF \cong \cO_Y^{mr}$. We will say that an $f$-Ulrich sheaf $\cF$ is positive if it is
$f$-positive.
To show the connection between  $f$-nonnegative sheaves and real fibered morphisms
we need a lemma:

\begin{lem} \label{lem:bilinear_ordering}
Let $L/K$ be a finite extension of fields, let $V$ be a finite-dimensional vector space over $L$.
Assume that there is an $L$-linear, non-zero homomorphism $\varphi \colon V \to \Hom_K(V,K)$,
such that the corresponding $K$-bilinear form on $V$ is symmetric 
and positive semidefinite with respect to some ordering $P$ on $K$, then $P$ has exactly $[L:K]$ extensions to $L$.
\end{lem}
\begin{proof}
After dividing out the kernel of $\varphi$ we can restrict to the case where $\varphi$ is injective.
Let $\alpha \in L$ be a primitive generator, i.e., $L = K[\alpha]$. Let $p$ be the minimal polynomial of $\alpha$. 
It suffices to show that $p$ splits over the real closure of $K$ with respect to $P$ that we will denote by $R$. 
Note that $p$ is the minimal polynomial of the $K$-linear map $f_{\alpha} \colon V \to V$ defined by $v \mapsto \alpha v$. 
Since $\varphi$ is $L$-linear, we have that $f_{\alpha}$ is selfadjoint with respect to the $K$-bilinear form defined by $\varphi$. 
Since the bilinear form induced on $V\ \otimes_K R$ is positive definite it admits an orthogonal basis. 
The representing matrix of $f_{\alpha}$ with respect to this basis is symmetric. Thus all of the roots of its characteristic polynomial lie in $R$. 
We conclude that $p$ splits over $R$.
\end{proof}

\begin{thm} \label{thm:positive_ulrich_then_real_fibered}
Let $X$ and $Y$ be irreducible $\R$-varieties and assume that $Y$ is smooth.
Let $f \colon X \to Y$ be a finite flat morphism. Then $f$ is real fibered if and only if there is an $f$-nonnegative coherent sheaf $\cF$ with
$\Supp \cF = X$.
\end{thm}
\begin{proof}
By Theorem \ref{thm:function_field_orede	} it suffices to show that $f$ restricted to an open dense subset is real fibered.
Thus by Grothendieck’s Generic
Freeness Lemma, we can assume without loss of generality that $X$ and $Y$ are affine and $\cF$ is a positive $f$-Ulrich sheaf.
Denote by $L$ the field of functions on $X$ and by $K$ the field of functions on $Y$.
Let us write $A = \cO_Y(Y)$, $B = \cO_X(X)$ and $M=\cF(X)$.
Represent the positive definite symmetric bilinear form on $A^{mr}$ by a positive definite symmetric $mr \times mr$ matrix $T$.
That means that the leading minors are positive at every closed point of $\Sper A$ and thus at every point of $\Sper A$.
Now the bilinear form translates to an isomorphism $M \to \Hom_B(M,\Hom_A(B,A)) \cong \Hom_A(M,A)$. Localizing at the generic point gives us
a bilinear form that satisfies the assumptions of Lemma \ref{lem:bilinear_ordering}. Thus $f$ is real fibered by Theorem \ref{thm:function_field_orede	}.

Conversely, let $f$ be real fibered. We will show that $\cF=\cO_X$ is $f$-nonnegative.
Since $f$ is flat and finite we can define the trace morphism $\textnormal{Tr}_{X/Y}: f_* \cO_X \to \cO_Y$.
This gives us a map $ f_* \cO_X \to \shHom_{\cO_Y}(f_* \cO_X, \cO_Y)$ defined as follows: 
For any open subset $U$ of $Y$ and any section $a \in \cO_X(f^{-1}(U))$ we define the image of $a$ to be the map $b \mapsto (\textnormal{Tr}_{X/Y})_U (ab)$,
cf. \cite[Ch. VI, \S 6]{AK70}.
The corresponding $\cO_Y$-valued bilinear form is positive semidefinite since $f$ is real fibered (cf. the remarks and references before Theorem
\ref{thm:function_field_orede	}) and by Grothendieck duality, it corresponds to an $f^! \cO_Y$-valued bilinear form on $\cO_X$.
Therefore, $\cO_X$ is $f$-nonnegative.
\end{proof}

\begin{rem}
 Theorem \ref{thm:positive_ulrich_then_real_fibered} includes the classic methods to check whether an univariate polynomial
 has only real roots or not (see for example \cite{KN81}) as special cases.
 The so-called Hermite matrices correspond to  $f^! \cO_Y$-valued bilinear forms  on the structure sheaf $\cO_X$ and the so-called B\'ezout matrices 
 correspond to  $f^! \cO_Y$-valued bilinear forms on the sheaf $f^! \cO_Y$, cf. \cite[Section 3]{Kum}.
\end{rem}

Theorem \ref{thm:positive_ulrich_then_real_fibered} says in particular that the existence of a positive $f$-Ulrich sheaf implies that $f$ is real fibered.
This motivates the following relative version of Question \ref{quest:generalized_lax}:

\begin{quest} \label{quest:relative_generalized_lax}
Let $f \colon X \to Y$ be a real fibered morphism of irreducible $\R$-varieties. Does there exist a real fibered morphism $g \colon Z \to Y$ and a closed embedding of $X$ into $Z$ over $Y$, such that $Z$ has a positive $g$-Ulrich sheaf?
\end{quest}

\begin{rem}
Note that this question differs from Question \ref{quest:generalized_lax}, since we do not require $X$ and $Z$ to be embedded into the same projective space and the maps to be linear projections. If we set $Y = \pp^k$, we get a relaxed version of Question \ref{quest:generalized_lax}, which is also open.
\end{rem}

%

Recall from Remark \ref{rem:chow} that the fact that $X$ has a positive definite admissible determinantal representation implies that the Chow form of $X$ has a determinantal representation in the following form:
\[
\det\left(\sum_{0 \leq i_0 < \ldots < i_k \leq d} p_{i_0, \ldots, i_d} A_j \right).
\]
Here the $A_j$ are constant Hermitian matrices, the $p_{i_0, \ldots, i_d}$ are the Pl\"{u}cker coordinates and at some real point on the Grassmannian the above matrix is definite. From Corollary \ref{cor:veronese_not_hyperbolic} we immediately get the following:

\begin{cor} \label{cor:chow_veronese}
Let $X = \cV_m(\pp^k)$, the Veronese embedding of $\pp^k$ into $\pp^N$, $k,m \geq 2$, where $N = \binom{k + m}{k} - 1$, then the Chow form of $X$, i.e., the resultant $k +1$ forms of degree $m$ in $k+1$ variables, does not have a representation as a determinant of a matrix of linear forms as above, where the $A_j$ are Hermitian and for some real point on the Grassmannian the above matrix is positive definite.
\end{cor}
\begin{proof}
If the Chow form of $X$ would have had such a determinantal representation it would be hyperbolic in the sense of \cite[Prop.\ 3.5]{SV14} and therefore $X$ itself would be hyperbolic contradicting Corollary \ref{cor:veronese_not_hyperbolic}.
\end{proof}

\begin{example}
 Recall from Example \ref{exp:interlace} that any two interlacing polynomials give a linear space with respect to which the rational normal curve is hyperbolic. Take for example $f=s^3-4s t^2$ and $g=s^2 t- t^3$. The morphism $\varphi$ then corresponds to the projection of the twisted cubic $C=\{(s^3:s^2 t:s t^2: t^3):\, (s:t) \in \mathbb{P}^1\}$ from the linear space defined by $x_0=4x_2$ and $x_1=x_3$. The matrix of the determinantal representation from \cite[Exp. 6.4]{SV14} for $C$ at this linear space is \[\begin{pmatrix}1&0&-1\\0&3&0\\-1&0&4\end{pmatrix}.\] This matrix is positive definite as expected.
\end{example}

\bibliographystyle{abbrv}
\bibliography{Real_Fibered_rev}

\end{document}

%% file: operators.tex
\usepackage{mathrsfs}

\newcommand{\coker}{\operatorname{coker}}

\newcommand{\Hom}{\operatorname{Hom}} 
\newcommand{\shHom}{\operatorname{{\mathscr Hom}}} 
\newcommand{\Spec}{\operatorname{Spec}}
\newcommand{\Sper}{\operatorname{Sper}}
\newcommand{\Supp}{\operatorname{Supp}}

\newcommand{\rank}{\operatorname{rank}} 

\newcommand{\diff}{\operatorname{d}}

\newcommand{\GL}{\operatorname{{\mathbf GL}}}

\newcommand{\tr}{\operatorname{tr}}


%% file: fonts.tex
\newcommand{\cE}{{\mathcal E}}
\newcommand{\cF}{{\mathcal F}}
\newcommand{\cG}{{\mathcal G}}

\newcommand{\cK}{{\mathcal K}}
\newcommand{\cL}{{\mathcal L}}

\newcommand{\cO}{{\mathcal O}}

\newcommand{\cV}{{\mathcal V}}

\newcommand{\G}{{\mathbb G}}
\newcommand{\C}{{\mathbb C}}
\newcommand{\R}{{\mathbb R}}
\newcommand{\pp}{\mathbb{P}}

\newcommand{\N}{{\mathbb N}}
\newcommand{\Z}{{\mathbb Z}}

